\documentclass[11pt]{amsart}
\usepackage{setspace}
\usepackage[margin=1in]{geometry}
\usepackage{mathrsfs}
\usepackage[T1]{fontenc}
\usepackage{xcolor}
\usepackage{algorithm}
\usepackage{algpseudocode}
\usepackage{bm}
\usepackage{amssymb}
\usepackage{multirow}
\usepackage{mathtools}
\usepackage[shortlabels]{enumitem}
\usepackage{capt-of}
\usepackage{booktabs}

\usepackage[utf8]{inputenc}
\usepackage[english]{babel}
\usepackage{graphbox}

\newtheorem{theorem}{Theorem}[section]
\newtheorem{corollary}{Corollary}[theorem]

\newtheorem{lemma}[theorem]{Lemma}

\theoremstyle{definition}

\theoremstyle{definition}
\newtheorem{remark}[theorem]{Remark}

\newcommand{\RN}[1]{%
	\textup{\uppercase\expandafter{\romannumeral#1}}%
}
\DeclareMathOperator*{\argmin}{\arg\!\min}
\newcommand{\bb}[1]{\mathbf{#1}}
\newcommand{\mmu}{\bm{\mu}}
\newcommand{\bba}{\bm{\alpha}}
\newcommand{\EE}[1]{\mathbb{E}\left[#1\right]}
\newcommand{\VV}[1]{\mathrm{Var}\left[#1\right]}
\newcommand{\CC}[2]{\mathrm{Cov}\left[#1, #2\right]}

\newcommand{\nn}[1]{\left|#1\right|}
\newcommand{\lr}[1]{\left(#1\right)}
\newcommand{\yb}{\bar{y}}
\newcommand{\yh}{\hat{y}}

\newcommand{\emf}{e\left(\hat{y}\right)}
\newcommand{\remf}{\hat{e}\left(\hat{y}\right)}
\newcommand{\emc}{e\left(\bar{y}^i_m\right)}
\newcommand{\remc}{\hat{e}\left(\bar{y}^i_m\right)}
\newcommand{\eMC}{e\left(\bar{y}^1_m\right)}

\begin{document}

\title{A Multifidelity Monte Carlo Method for Realistic Computational Budgets}

\author{Anthony Gruber$^{*,1}$ \and Max Gunzburger$^{1,2}$ \and Lili Ju$^3$ \and Zhu Wang$^3$}

\address{$^1$ Department of Scientific Computing, Florida State University, 400 Dirac Science Library, Tallahassee, FL 32306}

\address{$^2$ Oden Institute for Computational Engineering and Sciences, University of Texas at Austin, 201 E 24th Street, Austin, TX 78712}

\address{$^3$ Department of Mathematics, University of South Carolina, 1523 Greene Street, Columbia, SC 29208}

\thanks{$^*$Corresponding author: (Anthony Gruber)  agruber@fsu.edu}

\email{agruber@fsu.edu, mgunzburger@fsu.edu, ju@math.sc.edu, wangzhu@math.sc.edu}

\begin{abstract}
A method for the multifidelity Monte Carlo (MFMC) estimation of statistical quantities is proposed which is applicable to computational budgets of any size.  Based on a sequence of optimization problems each with a globally minimizing closed-form solution, this method extends the usability of a well known MFMC algorithm, recovering it when the computational budget is large enough. Theoretical results verify that the proposed approach is at least as optimal as its namesake and retains the benefits of multifidelity estimation with minimal assumptions on the budget or amount of available data, providing a notable reduction in variance over simple Monte Carlo estimation.
\end{abstract}

\maketitle

\section{Introduction}
Multifidelity Monte Carlo (MFMC) methods have shown great promise for the efficient and flexible estimation of statistical quantities.  As experimental data can take a variety of forms, the primary advantage of MFMC estimation is its ability to accommodate a diverse ensemble of information sources which may be unrelated apart from predicting the same quantity of interest.  However, despite their success in small proof-of-concept experiments where unlimited data are available, MFMC algorithms have not yet become standard in larger experiments of practical interest which are dominated by traditional Monte Carlo (MC) or Multilevel Monte Carlo (MLMC) techniques (see e.g. \cite{cubasch1994,gjerstad2003,leathers2004,hong2006,tomassini2007,mishra2012} and references therein).  While this is partly inertial, it is also true that MFMC algorithms are not always directly applicable to large scientific simulations, requiring nontrivial modification in order to function effectively.  A notable example of this is the popular MFMC estimation algorithm from \cite{peherstorfer2016} (hereafter known as MFMC), which divides computational resources through a global optimization with a unique closed-form solution. MFMC has been applied to great effectiveness on small-scale problems with functionally unlimited resources (e.g. \cite{patsialis2021,law2021,khodabakhshi2021}), but breaks down when the model costs are large relative to the available budget.  This is because the real-valued  optimization may call for some models to be evaluated less than once, producing an infeasible solution if integer parts are taken.  

Despite this technical concern, it is highly desirable to have an estimation method such as MFMC which is applicable to problems of any size and can guarantee good accuracy without access to inter-model pointwise error metrics.  To that end, this work provides a simple extension of the MFMC algorithm \cite[Algorithm 2]{peherstorfer2016} to any computational budget (c.f. Algorithm~\ref{alg:mod}), which reduces to the original procedure when the budget is large enough.  It is proved that this extension preserves the optimality of MFMC in a precise sense (c.f. Theorem~\ref{thm:optim}), and numerical examples are presented which illustrate this fact.  Applying the modified MFMC presented here to estimate the expectations of quantities arising from benchmark problems provides empirical justification for the proposed method, demonstrating that MFMC has benefits over traditional MC estimation even when the amount of available resources is quite limited.


Before discussing MFMC and its extension in more detail, it is appropriate to mention other concurrent work on multifidelity methods for optimization and uncertainty quantification (see \cite{peherstorfer2018survey} for an introductory survey).  About 20 years ago, a multifidelity method for function approximation involving neural-based systems was formulated as a kriging problem in \cite{leary2003}, while a multifidelity optimization method based on co-kriging was developed in \cite{forrester2007} and applied to Gaussian processes.  Later \cite{narayan2014} developed a multifidelity approach for stochastic collocation, and \cite{ng2014} proposed a multifidelity method for optimization under uncertainty which grew into the MFMC method of present discussion.  Note that the optimization approach of \cite{ng2014,peherstorfer2016} was also explored by other authors including \cite{pauli2015}, who proposed a branch-and-bound algorithm for optimal parameter selection in MLMC with applications to fault tolerance in large-scale systems.  Besides direct estimation, multifidelity methods have also found applications in preconditioning \cite{peherstorfer2019transport} as well as plasma micro-instability analysis \cite{konrad2022}.  At the time of writing, many multifidelity methods also have guarantees on their accuracy, including the MFMC method of present discussion which is analyzed in \cite{peherstorfer2018} and proven (under some assumptions) to be efficient in the sense that it requires a cost bounded by $1/\epsilon$ to achieve a mean squared error (MSE) of $\epsilon$.  It will be shown in Section~\ref{sec:mfmcnew} that the proposed modifications which enable realistic computational budgets also preserve this optimality (see Corollary~\ref{cor:optim}).  In preparation for this, the next Section reviews the MFMC algorithm of \cite{peherstorfer2016} in greater detail.

\section{Overview of MFMC}\label{sec:mfmcalg}
Let $\Omega$ be a sample space, $\mathcal{D}\subset\mathbb{R}^n$ for some $n\in\mathbb{N}$, and $f^1,f^2,...,f^k:\mathcal{D}\to\mathbb{R}$ be a sequence of computational models depending on the random variable $Z: \Omega\to \mathcal{D}$ with decreasing computational complexity.  The goal of MFMC is to estimate the expectation $y = \EE{f^1}$ using a linear combination of standard MC estimators.  Recall that Monte Carlo sampling yields the unbiased estimators 
\[ \bar{y}^i_{n} = \frac{1}{n}\sum_{j=1}^{n} f^i(\bb{z}_j), \]
where $\bb{z}_1,...,\mathbf{z}_{n}$ are i.i.d. realizations of $Z$ and unbiased implies that $\EE{\bar{y}^i_n} = \EE{f^i}$.  These can be used to define the MFMC estimator
\begin{equation}\label{eq:MFMC}
    \hat{y} = \bar{y}^1_{m_1} + \sum_{i=2}^k \alpha_i\left( \bar{y}^i_{m_i} - \bar{y}^i_{m_{i-1}}\right),
\end{equation}
which is unbiased by linearity and parameterized by weights $\alpha_i \in \mathbb{R}$ and an increasing sequence of integers $\{m_i\}$ defining the sampling.  If $w_i$ represents the cost of evaluating the $i^{th}$ model $f^i$, then the costs of computing the respective MC and MFMC estimators are given as
\begin{align*}
    c\left(\bar{y}^i_{n}\right) = nw_i, \qquad c(\hat{y}) = \sum w_im_i = \bb{w}\cdot\bb{m},
\end{align*}  
where $\bb{w}=\{w_i\}_{i=1}^k$ and $\bb{m}=\{m_i\}_{i=1}^k$.  For a fixed computational budget $p = \bb{w}\cdot\bb{m}\in\mathbb{R}$, MFMC aims to construct an optimal sampling strategy $\bb{m}$ with weights $\bm{\alpha} = \{\alpha_i\}_{i=2}^k$ so that the mean squared error (MSE) of the estimator $\hat{y}$ is lower than that of the Monte Carlo estimator $\bar{y}^1_n$ for all $n\leq \lfloor p/w_1\rfloor$.  More precisely, denote the variance of the $i^{th}$ model and the correlation between models $i,j$ by 
\[\sigma_i^2 = \VV{f^i(Z)}, \qquad \rho_{i,j} = \frac{\CC{f^i(Z)}{f^j(Z)}}{\sigma_i\sigma_j}.\] 
Then, the MSE of the MC estimator $\yb^i_n$ with respect to $y$ is simply
\[ e\lr{\bar{y}_n^i} =  \EE{\left(\EE{f^i(Z)} - \bar{y}^i_n\right)^2} = \EE{\left(\EE{\bar{y}_n^i} - \bar{y}^i_n\right)^2} = \frac{\sigma_i^2}{n}, \]
while the MSE of $\yh$ can be expressed (see \cite{peherstorfer2016}) as
\begin{equation}\label{eq:MSE}
    \emf = \VV{\hat{y}} = \frac{\sigma_1^2}{m_1} + \sum_{i=2}^k \left( \frac{1}{m_{i-1}} - \frac{1}{m_{i}} \right)\left( \alpha_i^2\sigma_i^2 - 2\alpha_i\sigma_i\sigma_1\rho_{1,i} \right).
\end{equation}
With this, the optimal weights $\bba^*=\{\alpha_i^*\}_{i=2}^k$ and sampling numbers $\bb{m}^*=\{m_i^*\}_{i=1}^k$ of the MFMC sampling strategy are defined through a solution to the mixed integer minimization problem (MIP),
\begin{equation}\label{eq:MFMCmin}
\begin{aligned}
\argmin_{\bb{m}\in\mathbb{Z}^k,\bba\in\mathbb{R}^{k-1}} \,\,& \emf \\
\textrm{s.t.} \quad
    &m_1 \geq 1, \\
    &m_{i-1} \leq m_i, \qquad 2\leq i \leq k, \\
    &\bb{w}\cdot\bb{m} \leq p.
\end{aligned}
\end{equation}

\begin{remark}
It is worth mentioning that the constraints in \eqref{eq:MFMCmin} are not the only options which lead to a potentially interesting MFMC strategy.  Since the MSE is generically a combination of variance and bias, the constraint $m_1\geq 1$ could be removed provided the expression $\emf$ is modified to include a bias term.  While this approach will not be investigated here, it remains an interesting avenue for future work.
\end{remark}

While analytic solutions to the problem  \eqref{eq:MFMCmin} are not easy to find (if they even exist), it turns out that the continuous relaxation obtained by letting $m_i\in\mathbb{R}^+$ for all $1\leq i \leq k$ has a closed-form solution (c.f. \cite[Theorem 3.4]{peherstorfer2016}).  In particular, suppose the computational models are decreasingly ordered with respect to their correlations with the high-fidelity model $1 = |\rho_{1,1}| > |\rho_{1,2}| > ... > |\rho_{1,k}|$ and have costs satisfying the inequality conditions 
\[ \frac{w_{i-1}}{w_i} > \frac{\rho_{1,i-1}^2 - \rho_{1,i}^2}{\rho_{1,i}^2 - \rho_{1,i+1}^2}, \qquad 2\leq i \leq k, \]
where $\rho_{1,k+1} := 0$; these conditions are always true for at least the subset $\{f^1\}$, in which case MFMC reduces to MC.  Introducing the notation
\[ r_i^* = \sqrt{ \frac{w_1\left( \rho_{1,i}^2 - \rho_{1,i+1}^2 \right)}{w_i\left( 1-\rho_{1,2}^2 \right)} }, \]
it follows from \cite[Theorem 3.4]{peherstorfer2016} that the optimal solution to the relaxation of \eqref{eq:MFMCmin} with $\bb{m}\in\mathbb{R}^{k}_+$ is given as
\begin{equation*}
    m_1^* = \frac{p}{\bb{w}\cdot\bb{r}}, \qquad
    \alpha_i^* = \frac{\rho_{1,i}\sigma_1}{\sigma_i}, \qquad m_i^* = m_1^* r_i^*, \qquad 2\leq i \leq k.
\end{equation*}
Temporarily neglecting the integrality of $\bb{m}^*$ necessary for evaluation, it follows that the MSE of the estimator $\hat{y}$ at this minimum becomes
\begin{equation}\label{eq:emfmc}
    e\left(\hat{y}\right) = \frac{\sigma_1^2\left(1-\rho_{1,2}^2 \right)p}{\left(m_1^*\right)^2 w_1},
\end{equation}
which is lower than the Monte Carlo equivalent with $c\left(\bar{y}^1_n\right) = p$,
\begin{equation}\label{eq:emc}
    e\left(\bar{y}^1_n\right)= \frac{\sigma_1^2}{n} = \frac{\sigma^2_1 w_1}{p},
\end{equation}
if and only if the ratio inequality
\begin{equation}\label{eq:compineq}
    \sqrt{\frac{\emf}{\eMC}}=\sum_{i=1}^k \sqrt{ \frac{w_i}{w_1}\left(\rho_{1,i}^2 - \rho_{1,i+1}^2\right) } < 1,
\end{equation}
is satisfied (c.f. \cite[Corollary 3.5]{peherstorfer2016}). This provides a way to explicitly compute the potential reduction in error provided by MFMC as a function of the weights and model correlations, and displays a primary advantage of this method: the mean squared error $\emf$ can be rigorously estimated without assumptions on the pointwise errors $\nn{f^1(\bb{z}) - f^i(\bb{z})}$ for $\bb{z}\in\mathcal{D}$ which may be unattainable in practice.

On the other hand, the optimal sampling numbers $\bb{m}^*$ computed according to the above procedure are not guaranteed to be integers.  The issues with this have seemingly not arisen before now, as the standard application of MFMC formulated in \cite[Algorithm 2]{peherstorfer2016} and widely used thereafter simply chooses the integer parts $\bb{m}^*\gets \lfloor\bb{m}^*\rfloor$ of each optimal sampling number.  For large budgets, this is of little consequence\textemdash each model is evaluated hundreds of times or more, so the integer parts $\lfloor m^*_i\rfloor$ provide a simple and reasonable solution which is guaranteed to be feasible (if potentially sub-optimal) for the original problem \eqref{eq:MFMCmin}.  However, this approach is unsuitable for the smaller budgets found in large-scale problems and simulated in Section~\ref{sec:numerics}, where $p \ll \bb{w}\cdot\bb{r}$ and some $m_i^* \ll 1$.  Taking the integer part $\lfloor m_1^*\rfloor$ in this case not only violates the constraints of the problem but also biases the estimator $\hat{y}$, destroying the applicability of MFMC.  The remainder of this work discusses a relatively slight modification which addresses this concern and is guaranteed to preserve optimality in the sampling strategy.

\section{A Modified MFMC Algorithm}\label{sec:mfmcnew}

Before discussing the proposed modification to MFMC which will apply to both small and large budgets, it is prudent to recall the two stages of the original MFMC algorithm developed in \cite{peherstorfer2016}.  Given the models $\{f^i\}_{i=1}^k$, the first stage outlined in Algorithm~\ref{alg:select} performs an exhaustive search over $2^{k-1}$ subsets $\mathcal{M}\subset\{f^i\}$ containing the high-fidelity model $f^1$, choosing the models which yield the lowest predicted MSE $\emf$ relative to the MC estimator.  Note that this choice is independent of the computational budget $p$ by \eqref{eq:compineq}, so it suffices to choose e.g. $p=10w_1$ as a benchmark value during model selection.  The second stage uses the result  $\mathcal{M}^*\subset \mathcal{M}$ of Algorithm~\ref{alg:select} to compute the optimal MFMC estimator $\hat{y}$.  More precisely, optimal weights $\bm{\alpha}^*$ and model evaluation numbers $\bb{m}^*$ which parameterize the best possible MFMC estimator $\hat{y}$ are first computed, so that samples of the random variable $Z$ can be drawn based on $\bb{m}^* \gets \lfloor\bb{m}^*\rfloor$ and used to compute $\hat{y}$ according to \eqref{eq:MFMC}.  However, this procedure breaks down when $p<\bb{w}\cdot\bb{r}$ as discussed before, making it necessary to find an alternative which is applicable when the budget $p$ is too small to produce an unbiased sampling strategy.  The most obvious ``solution'' to this problem is to na\"{i}vely compute $m_i^*\gets \lceil m_i^*\rceil$ whenever $m_i^*<1$, so that any sampling numbers less than one are rounded up. While this simple change will produce a meaningful estimator $\hat{y}$, it is not an acceptable modification in general, since it is likely that $\bb{w}\cdot\bb{m}^*>p$ exceeds the allowed computational budget (see e.g. the examples in  Section~\ref{sec:numerics}).

\begin{algorithm}[!ht]
\caption{MFMC Model Selection (c.f. \cite[Algorithm 1]{peherstorfer2016})}\label{alg:select}
\begin{algorithmic}[1]
    \Ensure $\rho_{1,i}^2 > \rho_{1,i+1}^2$ for $1\leq i \leq k$ (reorder if necessary).
    \State Set computational budget $p=w_1$.
    \State Initialize $\mathcal{M}^* = \{f^1\}$.
    \State Initialize $v^* = \sigma_1^2w_1/p$.
    \For{each $\mathcal{N}\subseteq \{f^2,...,f^k\}$ }
        \State Set $\mathcal{M} = \{f^1\} \cup \mathcal{N}$. 
        \State Set $k' = \nn{\mathcal{M}}$.
        \State Let $i_1,...,i_{k'}$ be the model indices of $\mathcal{M}$ such that $\rho_{1,i_j}^2 > \rho_{1,i_{j+1}}^2$ for $1\leq j \leq k'-1$.
        \State Set $\rho_{1,i_{k'+1}}=0$.
        \If{The following inequality holds:
        \[\frac{w_{i-1}}{w_i}\leq \frac{\rho_{1,i-1}^2-\rho_{1,i}^2}{\rho_{1,i}^2-\rho_{1,i+1}^2}, \quad\mathrm{for\,some\,\,} 2\leq i \leq k', \]\quad}
            \State Continue.
        \EndIf
        \State Compute the predicted variance 
        \[v = \frac{\sigma_1^2}{p}\lr{\sum_{j=1}^{k'}\sqrt{\rho_{1,i_j}^2-\rho_{1,i_{j+1}}^2}}^2.\]
        \If{$v<v^*$}
            \State Set $\mathcal{M}^*=\mathcal{M}$.
            \State Set $v^*=v$.
        \EndIf
        \EndFor
    \State \Return $\mathcal{M}^*,v^*$.
\end{algorithmic}
\end{algorithm}


While simple rounding is not enough to repair MFMC estimation when $p<\bb{w}\cdot\bb{r}$, it will now be shown that an appropriately modified MFMC algorithm always yields a solution which is constraint-preserving, feasible for \eqref{eq:MFMCmin}, and optimal up to the rounding $m_i^*\gets \lfloor m_i^*\rfloor$ of some components $m_i^*$.  First, note that it can be shown by repeating the arguments in \cite{peherstorfer2016} that the Lagrangian,
\begin{equation}\label{eq:ifunc}
    L^i\lr{\bb{m},\bba,\bm{\mu},\lambda,\xi} = \emf + \lambda \left(\sum_{j=i}^k w_jm_j - \left(p - \sum_{j=1}^{i-1} w_j\right)\right) + \sum_{j=i+1}^k\mu_j\left(m_j-m_{j-1}\right) - \xi m_i,
\end{equation}
has a unique global minimizer corresponding to a pair $(\bb{m}^*,\bba^*)\in\mathbb{R}^{k-i+1}\times\mathbb{R}^{k-i}$ for any $1\leq i \leq k$ (overloading the notation), where $\lambda\in\mathbb{R}, \mu_{i+1},...,\mu_k,\xi\in\mathbb{R}_+$ are dual variables encoding the constraints of \eqref{eq:MFMCmin}.  Observe that the minimization of $L^i$ for $i>1$ is almost the partial minimization of $L=L^1$ with the restriction $m_1=...=m_{i-1}=1$ on the model evaluation numbers, although they are not precisely equivalent.  Particularly, the $\lambda$-term enforces the budget constraint $\bb{w}\cdot\bb{m}=p$, and the $\mu_j$-terms enforce the increasing nature of the $m_j$ for $j\geq i$.  On the other hand, the constraint $m_i\geq m_{i-1}=1$ is not explicitly enforced, although the weaker constraint $m_i>0$ is still assured due to the $\xi$-term and the global minimizer of $L^i$ can be computed analytically for all $1\leq i \leq k$ using the following result.

\begin{algorithm}[!ht]
\caption{Modified MFMC Estimation}\label{alg:mod}
\begin{algorithmic}[1]
    \Ensure The models $f^1,...,f^k$ are the result of Algorithm~\ref{alg:select} and the budget satisfies $p \geq \sum_i w_i$.
    \State Set $\bm{\rho}_1 \gets \left[\bm{\rho}_1\,\, 0\right] \in \mathbb{R}^{k+1}$.
    \State Set the weights
    \[ \alpha_i = \frac{\rho_{1,i}\sigma_1}{\sigma_i},\qquad 2\leq i\leq k. \]
    \While{there is $1\leq i \leq k-1$ such that $m_i < 1$} 
        \State Set $i$ equal to the first such index.
        \State $m_i \gets 1$ and $r_{i+1} \gets 1$.
        \State Set the coefficients $r_j$, 
        \[r_j \gets \sqrt{\frac{w_{i+1}}{w_j}\left( \frac{\rho_{1,j}^2-\rho_{1,j+1}^2}{\rho_{1,i+1}^2 - \rho_{1,i+2}^2}\right)}, \qquad i+1<j\leq k.\]
        \State Set the $(i+1)$-st component of $\bb{m}$,
        \[m_{i+1} = \frac{p-\sum_{j=1}^i w_j}{\sum_{j=i+1}^k w_jr_j}.\]
        \State Set the remainder of $\bb{m}$,
        \[m_j = m_{i+1}r_j, \qquad i+1<j\leq k.\]
    \EndWhile
    \State Draw $m_k$ i.i.d. realizations $\bb{z}_1,...,\bb{z}_{m_k}$ of the random variable $Z$.
    \State Compute the MC estimators 
        \[ \bar{y}_{m_i}^i = \frac{1}{m_i}\sum_{j=1}^{m_i}f^i\lr{\bb{z}_j}, \qquad 1\leq i \leq k. \]
    \State \Return the MFMC estimator
        \[\hat{y} = \bar{y}^1_{m_1} + \sum_{i=2}^k \alpha_i\left( \bar{y}^i_{m_i} - \bar{y}^i_{m_{i-1}}\right).\]
\end{algorithmic}
\end{algorithm}

\begin{remark}
Note that the budget restriction $p\geq \sum_i w_i$ necessary for Algorithm~\ref{alg:mod} is the minimum necessary for the MFMC sampling procedure to make sense.  Indeed, if $p<\sum_i w_i$, then it is not possible to both evaluate $m_1$ and ensure that $\{m_i^*\}_{i=1}^k$ forms an increasing sequence.  Therefore, the claim that Algorithm~\ref{alg:mod} applies to any computational budget should be read with the understanding that these degenerate cases are omitted.
\end{remark}

\begin{lemma}\label{thm:mod}
Let $\{f^i\}_{i=1}^k$ be computational models with correlation coefficients satisfying $\nn{\rho_{1,i-1}}>\nn{\rho_{1,i}}$ for all $2\leq i \leq k$ and computational costs $\{w_1,...,w_k\}$ satisfying
\[ \frac{w_{i-1}}{w_i} > \frac{\rho_{1,i-1}^2-\rho_{1,i}^2}{\rho_{1,i}^2-\rho_{1,i+1}^2}, \qquad 2\leq i \leq k. \]
Defining the ratios 
\[ r^*_j = \frac{m_j^*}{m_i^*} = \sqrt{\frac{w_i}{w_j}\left( \frac{\rho_{1,j}^2-\rho_{1,j+1}^2}{\rho_{1,i}^2 - \rho_{1,i+1}^2}\right)},\qquad i\leq j \leq k, \]
for each $1\leq i \leq k$ the unique global minimum $\lr{\bb{m}^*,\bm{\alpha}^*}\in\mathbb{R}^{k-i+1}\times\mathbb{R}^{k-i}$ of the Lagrangian $L^i$ from \eqref{eq:ifunc} satisfies 
\begin{equation*}
    m_i^* = \frac{p-\sum_{j=1}^{i-1}w_j}{\sum_{j=i}^k w_jr_j}, \qquad m_j^* = m_i^* r^*_j, \qquad \alpha_j^* = \frac{\rho_{1,j}\sigma_1}{\sigma_j}, \qquad i<j\leq k.
\end{equation*}
\end{lemma}

\begin{proof}
Let $1\leq i \leq k$ and $2\leq j \leq k$.  A straightforward calculation using the representation \eqref{eq:MSE} gives the partial derivatives of the MSE,
\begin{align*}
    \emf_{m_i} &= \frac{1}{m_i^2}\left(\alpha_i^2\sigma_i^2-  2\alpha_i\rho_{1,i}\sigma_1\sigma_i\right) - \frac{1}{m_i^2}\left(\alpha_{i+1}^2\sigma_{i+1}^2-  2\alpha_{i+1}\rho_{1,i+1}\sigma_1\sigma_{i+1}\right), \\
    \emf_{\alpha_j} &= 2\left( \frac{1}{m_{j-1}}-\frac{1}{m_j}\right)\left(\alpha_j\sigma_j^2-\rho_{i,j}\sigma_1\sigma_j\right),
\end{align*}
so that for any $i$ the first-order conditions for the optimality of \eqref{eq:ifunc} are
\begin{equation}\label{eq:optim}
\begin{alignedat}{2}
    &L^i_{m_j} = \emf_{m_j}+\lambda w_j+\mu_{j+1}-\mu_j - \xi\delta_{ij} = 0, \qquad && i\leq j\leq k,\\
    &L^i_{\alpha_l} = \emf_{\alpha_l} = 0, \qquad && i<l\leq k, \\
    &\sum_{j=i}^kw_jm_j = p - \sum_{j=1}^{i-1}w_j, \qquad && \\
    &-\xi m_i = 0, \qquad && \\
    &\mu_l\left(m_{l-1}-m_l\right) \leq 0, \qquad && i<l\leq k, \\
    &m_{l-1}-m_l \leq 0, \qquad && i<l\leq k, \\
    &\xi,\mu_{i+1},...,\mu_k \geq 0,
\end{alignedat}
\end{equation}
where $\delta_{ij}$ denotes the Kronecker delta function.  The reader can check by repeating the arguments in \cite{peherstorfer2016} that $L^i$ has a unique global minimum $\lr{\bb{m}^*,\bm{\alpha}^*,\mmu^*,\lambda^*,\xi^*}$ satisfying $0<m^*_{j-1}<m^*_j$ for all $j>i$.  To compute the minimizing vectors $\bb{m}^*,\bm{\alpha}^*$, it suffices to manipulate the system \eqref{eq:optim}. It follows immediately from the equations $L^i_{\alpha_l} = 0$ that $\alpha^*_l = \rho_{i,l}\sigma_1/\sigma_l$.  Moreover, $\xi^*, \mmu^* = 0$ at the minimum since the entries of $\bb{m}^*$ form a strictly increasing sequence, so using $\bba^*$ in the equation $L^i_{m_i} = 0$ gives 
\[\lambda^* = \frac{\sigma_1^2\left(\rho_{1,i}^2-\rho_{1,i+1}^2\right)}{\lr{m_i^*}^2w_i},\]
which can be used in the remaining equations $L^i_{m_l}=0$ to derive the ratios,
\[ r_l^* = \frac{m_l^*}{m_i^*} = \sqrt{\frac{w_i}{w_l}\left( \frac{\rho_{1,l}^2-\rho_{1,l+1}^2}{\rho_{1,i}^2 - \rho_{1,i+1}^2}\right)}, \qquad i<l\leq k. \]
Extending this notation to include $r_i^*=1$ and using it alongside the fact 
\[\sum_{j=i}^k w_jr_j^* = \frac{1}{m_i^*}\sum_{j=i}^k w_jm_j^*,\]
then yields the expression
\[ m_i^* = \frac{p-\sum_{j=1}^{i-1}w_j}{\sum_{j=i}^k w_jr^*_j}, \]
from which the remaining components of $\bb{m}^*$ are computed through $m^*_l = m^*_ir^*_l$.
\end{proof}


This suggests the provably budget-preserving MFMC algorithm which is outlined in Algorithm~\ref{alg:mod}.  At the first step, the optimal model set $\{f^1,...,f^k\}$ is selected as in \cite[Algorithm 1]{peherstorfer2016} and the continuous relaxation of \eqref{eq:MFMCmin} is solved using Lemma~\ref{thm:mod} with $i=1$.  This generates a solution pair $\lr{\bb{m}^*,\bba^*}\in\mathbb{R}^{k}\times\mathbb{R}^{k-1}$ which may or may not contain components $m_i^*<1$.  If no component of $\bb{m}^*$ is less than one, then no modification is needed and simple rounding produces the model evaluation numbers $\bb{m}^*\gets\lfloor\bb{m}^*\rfloor$ which are nearly optimal and preserve the computational budget.  Otherwise, at least the first component $m_1^*<1$ due to the increasing order of $\bb{m}^*$, so the first component is redefined to $m_1^*=1$ and the function $L^2$ is minimized for the remaining $k-1$ components of $\bb{m}^*$.  This gives a new solution vector $\lr{m_2^*,...,m^*_k}$ which may or may not contain entries less than 1.  If it does not, the model evaluation vector becomes $\bb{m}^* = \left(1,\lfloor m^*_2\rfloor,...,\lfloor m^*_k\rfloor\right)$.  Otherwise, this process is repeated and the functions $L^i$ are minimized until there are no more indices $i$ satisfying $m^*_i<1$.  The result is a pair $\lr{\bb{m}^*,\bba^*}\in\mathbb{R}^{k}\times\mathbb{R}^{k-1}$ which represents a small perturbation from continuous optimality at each step and is nearly optimal for the original problem \eqref{eq:MFMCmin}.  In fact, there is the following guarantee.

\begin{theorem}\label{thm:optim}
Suppose $\bb{m}^*$ is the solution to the relaxation of \eqref{eq:MFMCmin} with $m^*_1<1$.  Then, $m^*_1 \gets 1$ is optimal for \eqref{eq:MFMCmin}.  Similarly, $m^*_i\gets 1$ is optimal for \eqref{eq:MFMCmin} whenever $m^*_1 = ... = m^*_{i-1}=1$ are fixed and $m^*_i<1$ is defined by the global minimizer of $L^i$.
\end{theorem}

\begin{proof}
Choose  $1 = \bar{m}_1 = ... = \bar{m}_{i-1}$ and consider varying $\bar{m}_i$.  Let $\lr{\bb{m},\bba} \in\mathbb{R}^{k-i}\times\mathbb{R}^{k-i-1}$ correspond to the unique global minimum of the Lagrangian $L^{i+1}$ from \eqref{eq:ifunc} for each $\bar{m}_i$,
\begin{equation*}
    m_{i+1} = \frac{p-\sum_{j=1}^{i}\bar{m}_jw_j}{\sum_{j=i+1}^k w_jr_j}, \qquad m_j = m_{i+1} r_j, \qquad \alpha_j = \frac{\rho_{1,j}\sigma_1}{\sigma_j}, \qquad i+1<j\leq k,
\end{equation*}
where $r_j$ is defined as in Lemma~\ref{thm:mod}.  Let $\tilde{\bb{m}}=\lr{\bar{m}_1,...,\bar{m}_i,m_{i+1},...,m_k}$ and consider the MSE at $\lr{\tilde{\bb{m}},\bba}$,
\begin{align*}
    \frac{\emf}{\sigma_1^2} &=  \frac{1}{\tilde{m}_1}-\sum_{j=2}^k\lr{\frac{1}{\tilde{m}_{j-1}}-\frac{1}{\tilde{m}_j}}\rho_{1,j}^2 \\ 
    &= \sum_{j=1}^{i-1}\frac{\rho_{1,j}^2-\rho_{1,j+1}^2}{\bar{m}_j}+\frac{\rho_{1,i}^2-\rho_{1,i+1}^2}{\bar{m}_{i}}+\frac{1}{m_{i+1}}\sum_{j=i+1}^k\frac{\rho_{1,j}^2-\rho_{1,j+1}^2}{r_j} =: f\lr{\bar{m}_i},
\end{align*}
where the first equality uses the definition of $\bba$ while the second equality follows by re-indexing and the fact that $m_j = m_{i+1}r_j$ for all $j\geq i+1$.  We will show that $f$ is a convex function of $\bar{m}_i$.  Notice that for $j\geq i+1$,
\[\frac{1}{r_j} = \sqrt{\frac{w_j\lr{\rho_{1,i+1}^2-\rho_{1,i+2}^2}}{w_{i+1}\lr{\rho_{1,j}^2-\rho_{1,j+1}^2}}} = \frac{w_j\lr{\rho_{1,i+1}^2-\rho_{1,i+2}^2}}{w_{i+1}\lr{\rho_{1,j}^2-\rho_{1,j+1}^2}}\sqrt{\frac{w_{i+1}\lr{\rho_{1,j}^2-\rho_{1,j+1}^2}}{w_j\lr{\rho_{1,i+1}^2-\rho_{1,i+2}^2}}}, \]
where the term under the square root on the right-hand side is precisely $r_j$.  It follows that 
\[ \sum_{j=i+1}^k\frac{\rho_{1,j}^2-\rho_{1,j+1}^2}{r_j} = \frac{\rho_{1,i+1}^2-\rho_{1,i+2}^2}{w_{i+1}}\sum_{j=i+1}^k w_jr_j,\]
and therefore $f$ has the form
\[ f\lr{\bar{m}_i} = C\lr{\bar{m}_1,...,\bar{m}_{i-1}} + \frac{\rho_{1,i}^2-\rho_{1,i+1}^2}{\bar{m}_{i}}+\frac{\rho_{1,i+1}^2-\rho_{1,i+2}^2}{w_{i+1}}\lr{\frac{\lr{\sum_{j=i+1}^k w_jr_j}^2}{p-\sum_{j=1}^i\bar{m}_jw_j}}, \]
where $C$ is constant in $\bar{m}_i$ and the definition of $m_{i+1}$ was used.  Now, differentiation in $\bar{m}_i$ yields 
\begin{align*}
   f'\lr{\bar{m}_i} &= \frac{w_i\lr{\rho_{1,i+1}^2-\rho_{1,i+2}^2}}{w_{i+1}}\lr{\frac{\sum_{j=i+1}^k w_jr_j}{p-\sum_{j=1}^i\bar{m}_jw_j}}^2 - \frac{\rho_{1,i}^2-\rho_{1,i+1}^2}{\bar{m}_{i}^2}, \\
    f''\lr{\bar{m}_i} &= \frac{2w_i^2\lr{\rho_{1,i+1}^2-\rho_{1,i+2}^2}}{w_{i+1}}\frac{\lr{\sum_{j=i+1}^k w_jr_j}^2}{\lr{p-\sum_{j=1}^i\bar{m}_jw_j}^3} + \frac{2\lr{\rho_{1,i}^2-\rho_{1,i+1}^2}}{\bar{m}_i^3} \\
    &\coloneqq C_1\lr{p-\sum_{j=1}^i\bar{m}_jw_j}^{-3} + C_2\bar{m}_i^{-3},
\end{align*}
where $C_1,C_2$ are constants independent of $\tilde{\bb{m}}$.  Clearly $f'' > 0$ whenever $\sum_{j\leq i}\bar{m}_jw_j<p$ (i.e. whenever the choice of $\bar{m}_i$ is feasible for the budget), so that the MSE $\emf$ is convex in $\bar{m}_i$.  Therefore, $\bar{m}_i=1$ must be the optimal integer choice whenever $\bar{m}_i=m_i<1$ is defined by the global minimizer of $L^i$.
\end{proof}

Theorem~\ref{thm:optim} provides a guarantee that the modification in Algorithm~\ref{alg:mod} is optimal for the original MIP \eqref{eq:MFMCmin}, and leads to the immediate conclusion that the modified MFMC Algorithm~\ref{alg:mod} is at least as optimal for \eqref{eq:MFMCmin} as the original algorithm from \cite{peherstorfer2016}.  Indeed, the above shows that $m^*_1 \gets 1$ is the integer value leading to the smallest $\emf$ whenever $m_1^*<1$ solves the relaxation of \eqref{eq:MFMCmin}.  Similarly, if $m_1^*=...=m_{i-1}^*=1$ and $m_i^*<1$ minimizes $L^i$, then $m_i^*\gets 1$ leads to the smallest $\emf$ and satisfies the constraints of the MIP \eqref{eq:MFMCmin}.  Repeating this sequentially until there are no more $m^*_i<1$ immediately yields the following Corollary.

\begin{corollary}\label{cor:optim}
The modification to MFMC in Algorithm~\ref{alg:mod} is optimal for \eqref{eq:MFMCmin} up to the rounding $\bb{m}^*\gets\lfloor\bb{m}^*\rfloor$ inherent in the original procedure.
\end{corollary}

Corollary~\ref{cor:optim} is analogous to the similar guarantees in \cite{peherstorfer2016,peherstorfer2018} which establish optimality of the model evaluation vector $\bb{m}^*$ for the relaxation of \eqref{eq:MFMCmin}.  On the other hand, the above result still cannot guarantee that the coefficients $m_i^*>1$ which solve the relaxed problem are optimal for the MIP \eqref{eq:MFMCmin}, although it is possible to find a small set of candidates which can be manually checked for optimality.  Indeed, the proof above shows that the MSE  $\emf$ is convex in the sub-leading entry $m^*_i$ of the global minimizer of $L^{i+1}$, so that given the minimizer for $L^i$ at each step the optimal choice for $m_i^*$ is either $\lfloor m_i^*\rfloor$ or $\lceil m_i^*\rceil$.  Therefore, a branch-and-bound algorithm based on this fact is guaranteed to produce a global minimum for \eqref{eq:MFMCmin} in at most $2^k$ steps.  As the simpler Algorithm~\ref{alg:mod} is already accurate enough for the present purposes, an approach based on this idea is left for future work.

\section{Numerical Experiments}\label{sec:numerics}

The results of the previous Section are now tested via two benchmark experiments which simulate low-budget use cases for MC and MFMC estimation.  More precisely, given computational models $\{f^i\}_{i=1}^k$ selected by Algorithm~\ref{alg:select} the modified MFMC Algorithm~\ref{alg:mod} is used to estimate the expectation of scalar quantities which arise from small-scale benchmark simulations in mechanics.  The performance of this estimation is evaluated for a range of computational budgets through comparison with simple MC estimation and a na\"{i}ve (generally infeasible) version of the original MFMC \cite[Algorithm 2]{peherstorfer2016}, whereby the optimal sampling $\bb{m}^*\in\mathbb{R}_+^k$ is rounded according to the rule
\[ m_i^* \gets \begin{cases}\lfloor m_i^*\rfloor & m_i^*\geq 1, \\
\lceil m_i^*\rceil & m_i^*<1\end{cases}. \]
Although this usually produces a sampling which exceeds the budget and is therefore impractical, its inclusion  provides a useful link between the modifications of Section~\ref{sec:mfmcnew} and the original MFMC algorithm.



The performance of each estimator $\tilde{y}$ produced by either MC or MFMC is measured using the MSE (or variance) $e\lr{\tilde{y}}$ and the relative MSE (or relative variance) $\hat{e}\lr{\tilde{y}} \coloneqq e\lr{\tilde{y}}/\lr{y_{\mathrm{ref}}^2}$, which provide a measure of spread around the desired expectation.  Since evaluating these MSEs requires an ``exact'' value $y_{\mathrm{ref}} \approx \EE{f^1}$ for comparison, for each example an MC approximation to $\EE{f^1}$ is constructed using $M>0$ i.i.d. samples of $Z$ for some large $M$, 
\[y_{\mathrm{ref}} = \frac{1}{M}\sum_{m=1}^{M}f^1\lr{\bb{z}_m}.\]
Moreover, to further mitigate inherent dependence on stochasticity during these examples, the expectations and MSEs generated by MC and MFMC are averaged over $N>0$ independent runs of each method, so that the estimators and MSEs used in practice are computed as
\[\bar{y} = \frac{1}{N}\sum_{j=1}^Ny_j, \qquad e\lr{\bar{y}} = \frac{1}{N}\sum_{j=1}^N \left(y_j - y_{\mathrm{ref}}\right)^2, \]
where $y_j \approx \EE{f^1}$ denotes the approximator produced by the relevant method (MC or MFMC) at run $j$.  Empirical results show that this facilitates a comparison which is less sensitive to random sampling than if each method is carried out only once (c.f. Figure~\ref{fig:toy}, bottom row).  Finally, note that all experiments in this Section are carried out in Python 3.9 on a 2022 Macbook Pro with 32GB of RAM using a NumPy \cite{harris2020} implementation of Algorithms~\ref{alg:select} and \ref{alg:mod}.

\subsection{Analytical model of a short column}
The first example considered here is taken from \cite{peherstorfer2016} and involves an analytic model of a short column with rectangular cross-sectional area subject to bending and axial force.  The high-fidelity model is
\[ f^1(\bb{z}) = 1 - \frac{4z_4}{z_1z_2^2z_3} - \left(\frac{z_5}{z_1z_2z_3}\right)^2,\]
where $\bb{z} \in [5,15]\times[15,25]\times\mathbb{R}^+\times\mathbb{R}^2 \subset \mathbb{R}^5$.  The random variable $Z$ is chosen such that the width and depth $z_1,z_2$ are distributed uniformly, the yield stress $z_3$ is distributed log-normally with mean 5 and standard deviation 0.5, the bending moment $z_4$ is distributed normally with mean 2000 and standard deviation 400, and the axial force $z_5$ is distributed normally with mean 500 and standard deviation 100.  Since the computational cost of $f^1$ is arbitrarily, the low-fidelity surrogates and computational costs are chosen semi-arbitrarily. In particular, the surrogates are taken to be
\begin{align*}
    f^2(\bb{z}) &= 1 - \frac{3.8z_4}{z_1z_2^2z_3} - \left(\frac{z_5\left(1 + \frac{z_4-2000}{4000}\right)}{z_1z_2z_3}\right)^2, \\
    f^3(\bb{z}) &= 1 - \frac{z_4}{z_1z_2^2z_3} - \left(\frac{z_5(1+z_4)}{z_2z_3}\right)^2, \\
    f^4(\bb{z}) &= 1 - \frac{z_4}{z_1z_2^2z_3} - \left(\frac{z_5(1+z_4)}{z_1z_2z_3}\right)^2, \\
    f^5(\bb{z}) &= 1 - \frac{z_4}{z_1z_2^2z_3} - \left(\frac{z_5}{z_1z_2z_3}\right)^2,
\end{align*}
where we note that $f^3$ and $f^5$ are reversed from their definitions in \cite{peherstorfer2016} (and this will not affect the results as the surrogates are always reordered by correlation).  Denoting $\boldsymbol{\rho}_1 = \{\rho_{1,i}\}_{i=1}^k$, the weights and model correlation coefficients are 
\begin{align*}
    \bb{w} &= \begin{pmatrix} 100 & 50 & 20 & 10 & 5 \end{pmatrix}^\intercal,\\
    \boldsymbol{\rho}_1 &= \begin{pmatrix} 1.0000000 & 0.99994645 & 0.6980721 & 0.92928154 & 0.99863737 \end{pmatrix}^\intercal,
\end{align*}
where 1000 i.i.d. samples of $Z$ are used to compute the displayed correlations.  With this data, the model selection Algorithm~\ref{alg:select} produces the subset $\mathcal{M}^*=\{f^1,f^2,f^5\}$ to be used in constructing the MFMC estimator $\hat{y}\approx\EE{f^1}$.  Here, the reference value $y_\mathrm{ref}$ is computed from $M=10^7$ i.i.d. samples of $Z$ and the estimators/MSEs are averaged over $N=1000$ independent runs.  For comparison, a visual of the $N=1$ run case is provided in the bottom row of Figure~\ref{fig:toy}.

\begin{table}[!ht]
\caption{Results of the short column experiment for various computational budgets $p$.  All errors are reported in units of $10^{-6}$.}\small
\smallskip
\centering
\begin{tabular}{ccccccccc}\toprule
\multicolumn{3}{c}{MC} & \multicolumn{6}{c}{Modified MFMC (Rounded MFMC)}
\\\cmidrule(lr){1-3}\cmidrule(lr){4-9}
           $p/w_1$ & $\remc$ & $\emc$ & $p/w_1$ & \# $f^1$ & \# $f^2$ & \# $f^5$ & $\remf$ & $\emf$ \\\midrule
2 & 25.94 & 25.29 & 2 (3.15) & 1 (1) & 1 (1) & 10 (33) & 5.020 (1.496) & 4.893 (1.458) \\
4 & 11.47 & 11.18 & 4 (4.8) & 1 (1) & 1 (1) & 50 (66) & 1.082 (0.8865) & 1.055 (0.8641) \\
8 & 5.860 & 5.712 & 7.5 (8.6) & 1 (1) & 1 (2) & 120 (132) & 0.4986 (0.4329) & 0.4860 (0.4219) \\
16 & 2.911 & 2.837 & 15.9 (16.2) & 1 (1) & 4 (4) & 258 (264) & 0.2434 (0.2159) & 0.2372 (0.2104) \\
32 & 1.463 & 1.426 & 31.45 (31.45) & 1 (1) & 8 (8) & 529 (529) & 0.1059 (0.1167) & 0.1032 (0.1137) \\
64 & 0.7411 & 0.7223 & 63.45 (63.45) & 2 (2) & 17 (17) & 1059 (1059) & 0.05698 (0.05485) & 0.05554 (0.05346) \\\bottomrule
\end{tabular}
\label{tab:toy}
\end{table}


Table~\ref{tab:toy} gives a numerical summary of the comparison between MC and both variants of MFMC when estimating the expectation $y=\EE{f^1}$, showing explicitly that the sampling Algorithms~\ref{alg:mod} and \cite[Algorithm 2]{peherstorfer2016} are generally inequivalent but reduce to the same procedure for large enough computational budgets.  Figure~\ref{fig:toy} plots the values resp. relative MSEs of the estimators obtained from MC and each MFMC method, along with shading indicating the standard deviation resp. variance in these quantities over the $N=1000$ runs.  As expected, each MFMC procedure is sufficient for a variance reduction of about an order of magnitude relative to MC regardless of the available budget.  However, the na\"{i}ve rounding procedure necessary for a direct application of \cite[Algorithm 2]{peherstorfer2016} clearly exceeds the budget in most cases, rendering it unreliable as an estimation method when computational resources are limited.  Despite this technical difficulty which has been addressed in Algorithm~\ref{alg:mod}, it is clear that MFMC estimation has benefits over traditional MC when a variety of information sources e.g. surrogate models are available.  Since MFMC is able to utilize more than just the high-fidelity model $f^1$, much of the computational budget can be loaded onto the cheaper surrogate models in order to decrease the MSE of the estimator $\hat{y}$.


\begin{figure}[!ht]
    \centering
    \begin{minipage}{0.33\textwidth}
        \includegraphics[width=\textwidth]{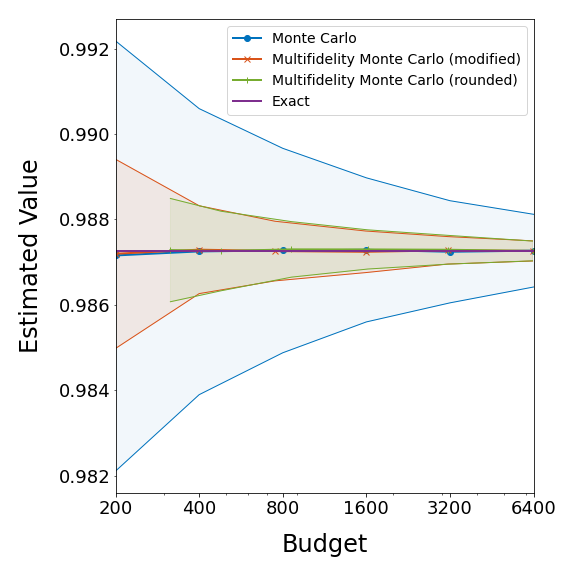}
    \end{minipage}%
    \begin{minipage}{0.33\textwidth}
        \includegraphics[width=\textwidth]{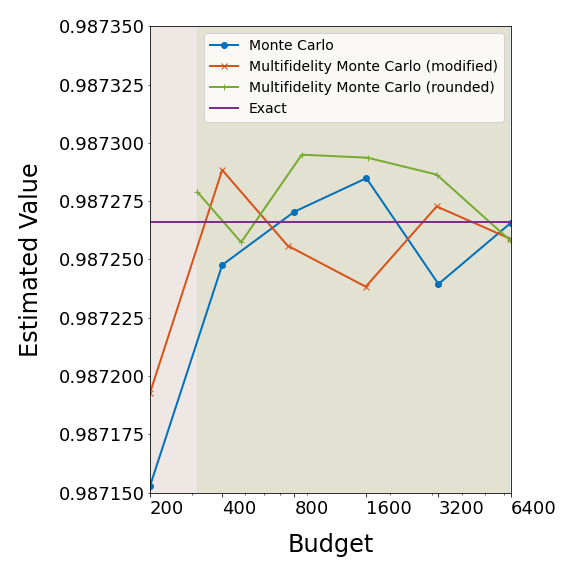}
    \end{minipage}%
    \begin{minipage}{0.33\textwidth}
        \centering
        \includegraphics[width=\textwidth]{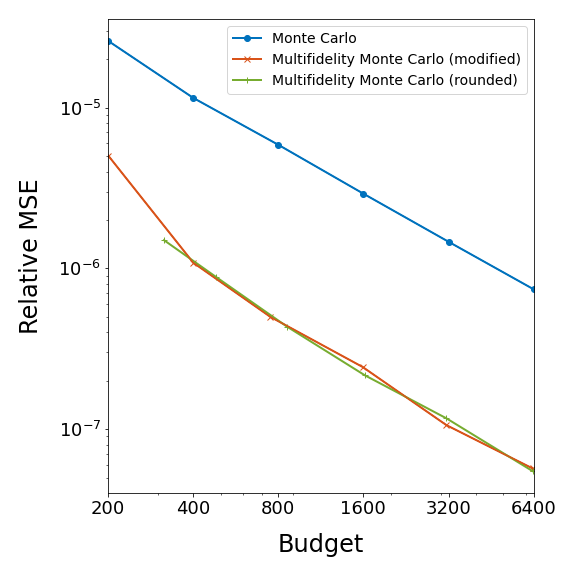}
    \end{minipage} \\
    \begin{minipage}{0.33\textwidth}
        \includegraphics[width=\textwidth]{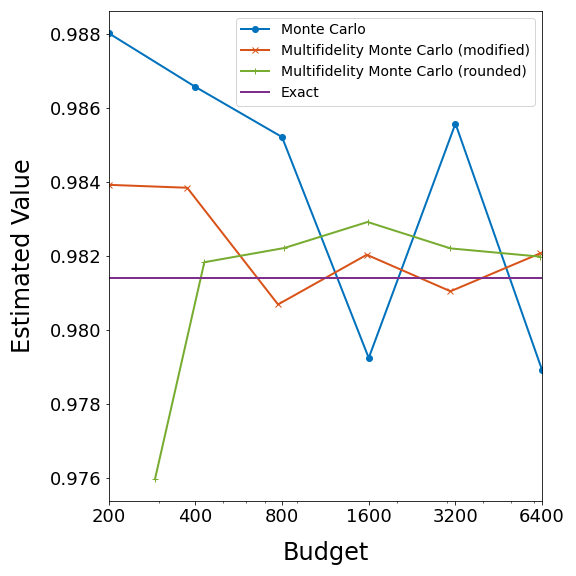}
    \end{minipage}%
    \begin{minipage}{0.33\textwidth}
        \includegraphics[width=\textwidth]{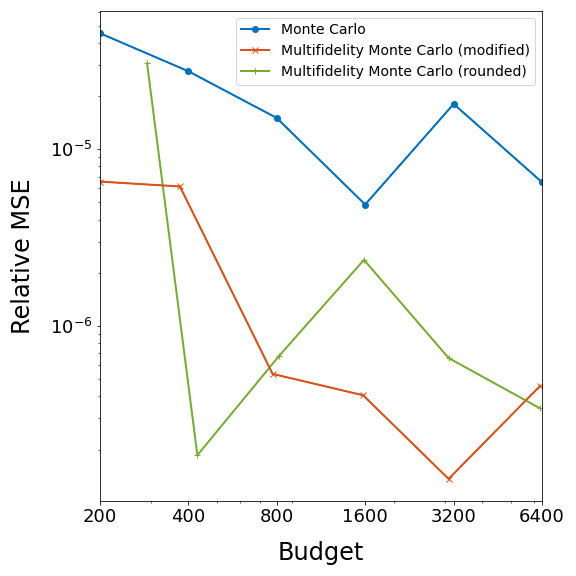}
    \end{minipage}
    \caption{Pictures corresponding to the short column experiment.  Plotted in the top row (left, middle) are the mean estimates and (right) MSEs over $N=1000$ independent runs of each method, along with shaded regions representing the standard deviations (left, middle) and variances (right, imperceptible) in these quantities over the $N$ runs. For comparison, the bottom row displays these same quantities for $N=1$ run of each method. }
    \label{fig:toy}
\end{figure}

\subsection{Inviscid Burgers equation}

The next experiment involves a parameterized version of the Inviscid Burgers equation seen in  \cite{gruber2022comparison}, which is a common model for systems which experience shocks.  Let $x\in[0,100]$, $t\in[0,10]$, and $\bb{z} = \begin{pmatrix}z_1 & z_2\end{pmatrix}^\intercal \in [0.5,3.5]\times [2\times 10^{-4},2\times 10^{-3}]$ be a vector of parameters.  Consider the initial value problem,
\begin{align}\label{eq:burgers}
\begin{split}
    w_t + \frac{1}{2}\left(w^2\right)_x &= 0.02 e^{z_2 x}, \\
    w(0, t, \bb{z}) &= z_1, \\
    w(x,0,\bb{z}) &= 1,
\end{split}
\end{align}
where subscripts denote partial differentiation and $w=w(x,t, \bb{z})$ represents the conserved quantity of the fluid.  Discretizing the spatial interval $[0,100]$ with finite differences converts \eqref{eq:burgers} to a system of $n$ ordinary differential equations defined at the nodal points $\{x^\alpha\}_{\alpha=1}^{n}$, which can be further discretized in time and solved for any $\bb{z}$ using a simple upwind forward Euler scheme.  More precisely, let $\Delta x = 100/256, \Delta t = 1/10$ be stepsizes and let $w^\alpha_k$ denote the solution $w$ at time step $k$ and spatial node $\alpha$.  Then, forward differencing applied to \eqref{eq:burgers} gives the algebraic system
\begin{equation*}
    \frac{w_{k+1}^\alpha - w_k^\alpha}{\Delta t} + \frac{1}{2}\frac{\left(w^2\right)_k^\alpha - \left(w^2\right)_k^{\alpha-1}}{\Delta x} = 0.02e^{z_2 x}, \qquad 1\leq\alpha\leq n
\end{equation*}
which expresses the solution $w^\alpha(t,\bb{z})\coloneqq w(x^\alpha,t,\bb{z})$ at time $t = (k+1)\Delta t$ in terms of its value at time $t = k\Delta t$.  Collecting this for each $1\leq \alpha\leq n$ yields the vectorized solution $\bb{w}(t,\bb{z}) \in \mathbb{R}^n$ which is stored in practice and can be used to define a quantity of interest for statistical estimation.  In particular, let $Z$ be a random variable chosen so that its realizations $(z_1,z_2)$ are distributed uniformly in their ranges of definition.  Then, $\bb{w} = \bb{w}(Z)$ becomes a random variable and it is meaningful to consider the high-fidelity model
\[ f^1\lr{\bb{w}(Z)} = \frac{1}{n}\sum_{\alpha=1}^n w^\alpha_{T}, \]
where the notation $w^\alpha_{T}$ indicates the value of the solution at node $\alpha$ and final time $T=10$.  This gives an average measure of $w$ which is useful to track as a function of the random variable $Z$ and whose expectation $y = \EE{f^1}$ can be estimated using MC and MFMC.  

To simulate a low-budget use case for MC/MFMC estimation in this context, access to the high-fidelity model $f^1$ is assumed along with an ensemble of reduced-order models (ROMs) based on proper orthogonal decomposition (POD) which will now be described.  Recall that POD uses snapshots of the high-fidelity solution at various time instances to generate a reduced basis of size $d \ll n$ whose span has minimal $\ell_2$ reconstruction error.  For this example, snapshot matrices $\bb{S} = \left(s^\alpha_j\right)$ with entries $s^\alpha_j = w^\alpha(t_j)$ are collected from simulations corresponding to 50 uniform i.i.d. samples of $Z$ and preprocessed by subtracting the relevant initial condition from their columns, generating matrices $\bar{\bb{S}} = \left(\bar{s}^i_j\right)$ with $\bar{s}^i_j = s^i_j - w^i(t_0)$.  These 50 matrices are then concatenated column-wise to form the full snapshot matrix $\bar{\bb{S}} = \bb{U}\bm{\Sigma}\bb{V}^\intercal$ (overloading notation), so that the optimal POD basis of dimension $d$ is contained in the matrix $\bb{U}_d$ containing the first $d$ columns of $\bb{U}$.  Denoting $\bb{w}_0 = \bb{w}(t_0)$ and $\hat{\bb{w}}\in\mathbb{R}^d$ a coefficient vector, this enables a POD approximation to the high-fidelity solution $\tilde{\bb{w}} = \bb{w}_0 + \bb{U}_d\hat{\bb{w}}, \approx \bb{w}$ which can be substituted into the discretization of \eqref{eq:burgers} to yield the reduced-order system
\begin{equation}\label{eq:burgersROM}
    \begin{split}
        \frac{\hat{\bb{w}}_{k+1}-\hat{\bb{w}}_k}{\Delta t} + \frac{1}{2}\left(\bb{a} + \bb{B}\hat{\bb{w}}_k + \bb{C}(\hat{\bb{w}}_k,\hat{\bb{w}}_k)\right) &=  0.02 \bb{U}_d^\intercal e^{z_2\bb{x}}, \\
        \hat{\bb{w}}_0 &= \bb{U}_d^\intercal(\bb{w}_0-\bb{w}_0) = \bm{0}.
    \end{split}
\end{equation}
where $0 \leq k \leq N_t-1$, $e^{z_2\bb{x}}$ acts component-wise, and $\bb{a},\bb{B},\bb{C}$ are vector-valued quantities which can be precomputed.  To describe this more explicitly, let $\bb{D}_x \in \mathbb{R}^{N\times N}$ denote the bidiagonal forward difference matrix with entries $\{0,1/\Delta x,...,1/\Delta x\}$ on its diagonal and $\{-1/\Delta x,..,-1/\Delta x\}$ on its subdiagonal. A straightforward computation then yields 
\begin{align*}
\bb{a} &= \bb{U}_d^\intercal\bb{D}_x(\bb{w}_0)^2, \\
\bb{B} &= 2\bb{U}_d^\intercal\bb{D}_x\mathrm{Diag}(\bb{w}_0)\bb{U}_d, \\
\bb{C} &= \bb{U}_d^\intercal\bb{D}_x \mathrm{diag}_{1,3}(\bb{U}_d \otimes \bb{U}_d),
\end{align*}
where $(\bb{w}_0)^2$ is a component-wise operation taking place in $\mathbb{R}^N$, $\mathrm{Diag}(\bb{w}_0)$ indicates the diagonal matrix built from $\bb{w}_0$, and $\mathrm{diag}_{1,3}(\bb{U}_d\otimes\bb{U}_d)$ indicates the diagonal of $\bb{U}_d\otimes\bb{U}_d$ that occurs when components 1 and 3 are the same, i.e. the rank 3 tensor with components $u^\alpha_j u^\alpha_k$.  Solving the $d$-dimensional system \eqref{eq:burgersROM} for $\hat{\bb{w}}$ gives an approximate solution $\tilde{\bb{w}}$ which is computationally less expensive and can be made arbitrarily accurate by increasing the  dimension $d$.  

\begin{figure}[!ht]
    \centering
    \begin{minipage}{0.33\textwidth}
        \includegraphics[width=\textwidth]{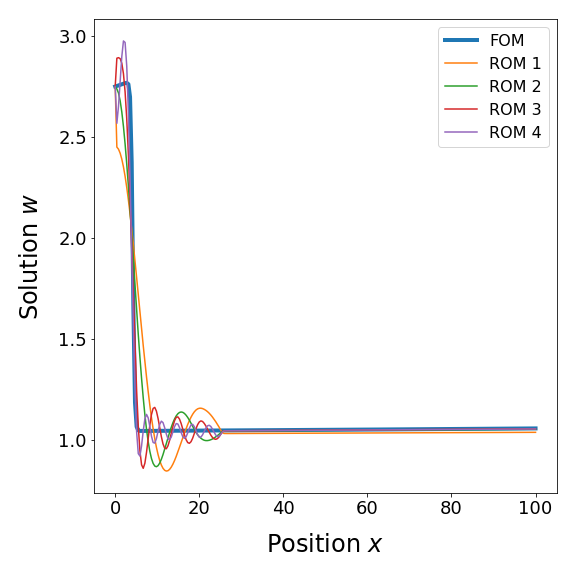}
    \end{minipage}%
    \begin{minipage}{0.33\textwidth}
        \includegraphics[width=\textwidth]{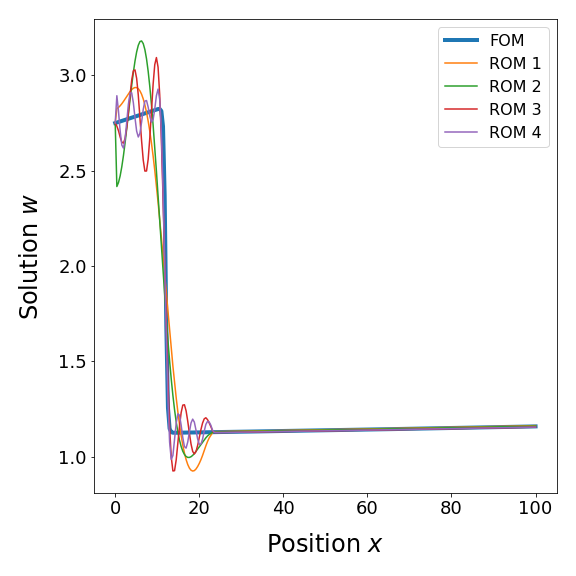}
    \end{minipage}%
    \begin{minipage}{0.33\textwidth}
        \centering
        \includegraphics[width=\textwidth]{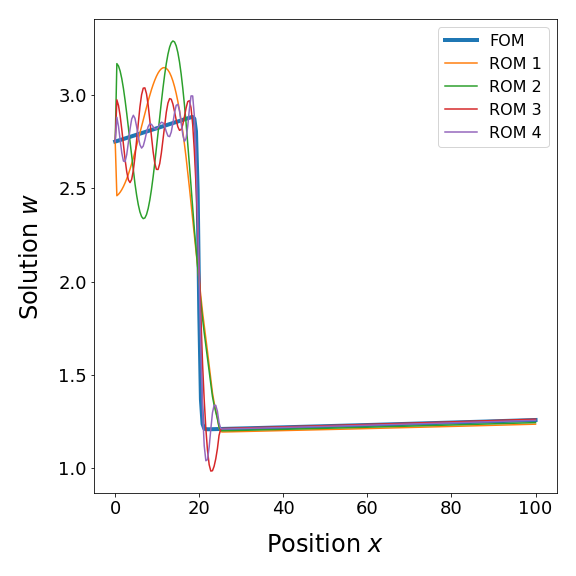}
    \end{minipage}
    \caption{Snapshots of the high-fidelity PDE model and its POD-ROM surrogates corresponding to the inviscid Burgers example at times $t = 2,6,10$ and parameter value $\bb{z} = \begin{pmatrix}2.75 & 0.0275\end{pmatrix}^\intercal$.}
    \label{fig:burgersFOM}
\end{figure}

In view of this, the ROM \eqref{eq:burgersROM} can be used to obtain multifidelity surrogates for $f^1(Z)$ by varying the approximation quality of $\tilde{\bb{w}}$.  The present experiment considers a sequence of ROMs defined by $d = \{3, 5, 10, 15\}$ and their associated surrogate models $\{f^2,f^3,f^4,f^5\}$ defined analogously to $f^1$.  It is evident from Figure~\ref{fig:burgersFOM} that approximation quality increases with $d$, although this benefit comes with an increased computational cost.  Using 100 i.i.d. samples of $Z$ to compute the computational costs and inter-model correlations for the model set $\{f^i\}_{i=1}^5$ yields the data
\begin{align*}
    \bb{w} &= 10^{-4}\begin{pmatrix} 30.5625 & 5.5174 & 5.8633 & 6.3854 & 7.4522 \end{pmatrix}^\intercal,\\
    \boldsymbol{\rho}_1 &= \begin{pmatrix} 1.0000000 & 0.99766585 & 0.98343683 & 0.99999507 & 0.99999882 \end{pmatrix}^\intercal,
\end{align*}
where the cost vector $\bb{w}$ is computed as the average time (in seconds) necessary to evaluate each model.  From this, the model selection Algorithm~\ref{alg:select} chooses the subset $\{f^1, f^4, f^2\}$ for use in constructing the MFMC estimator $\hat{y}$ which will approximate $y=\EE{f^1}$.  To evaluate the quality of MC and MFMC estimation, a reference MC approximation $y_\mathrm{ref}\approx \EE{f^1}$ is generated from $10^5$ i.i.d. samples of $Z$, and each estimator/MSE computed is averaged over $N=100$ independent runs of its relevant method.

\begin{table}[!ht]
\caption{Results of the Inviscid Burgers experiment for various computational budgets $p$.  All errors are reported in units of $10^{-3}$.}\small
\smallskip
\centering
\begin{tabular}{ccccccccc}\toprule
\multicolumn{3}{c}{MC} & \multicolumn{6}{c}{Modified MFMC (Rounded MFMC)}
\\\cmidrule(lr){1-3}\cmidrule(lr){4-9}
           $p/w_1$ & $\remc$ & $\emc$ &   $p/w_1$ & \# $f^1$ & \# $f^4$ & \# $f^2$ & $\remf$ & $\emf$ \\\midrule
2 & 17.35 & 36.96 & 1.93 (3.01) & 1 (1) & 1 (1) & 4 (10) & 8.077 (2.404) &  17.20 (5.120) \\
4 & 8.256 & 17.20 & 3.92 (4.82) & 1 (1) & 1 (1) & 15 (20) & 1.745 (1.493) & 3.716 (3.181) \\
8 & 3.428 & 7.302 & 7.92 (8.64) & 1 (1) & 2 (2) & 36 (40) & 1.145 (0.9163) & 2.440 (1.952) \\
16 & 1.852 & 3.945 & 15.7 (16.7) & 1 (1) & 4 (5) & 77 (81) & 0.3487 (0.2590) & 0.7427 (0.5517) \\
32 & 0.9291 & 1.979 & 31.8 (32.5) & 1 (1) & 10 (10) & 159 (163) & 0.1713 (0.1588) & 0.3649 (0.3383) \\
64 & 0.4641 & 0.9885 & 63.9 (64.2) & 1 (1) & 20 (20) & 325 (327) & 0.07745 (0.1039) & 0.1651 (0.2213) \\\bottomrule
\end{tabular}
\label{tab:burgers}
\end{table}

The results of this experiment are displayed in Table~\ref{tab:burgers} and Figure~\ref{fig:burgersROM}.  Again it is clear that simply rounding the solution of \cite[Algorithm 2]{peherstorfer2016} is not sufficient for producing a feasible MFMC estimator in the presence of limited computational resources, as the target budget is exceeded (sometimes greatly so) in every case.  This confirms that the modified MFMC Algorithm~\ref{alg:mod} is necessary for applying MFMC estimation in these cases, and gives confidence that Algorithm~\ref{alg:mod} will also be effective on large-scale problems where such resource restrictions are commonplace.  Moreover, it is also evident from the results presented here that Algorithm~\ref{alg:mod} preserves the accuracy benefits of MFMC estimation over simple MC, producing a notably lower MSE regardless of the size of computational budget.

\begin{figure}[!ht]
    \centering
    \begin{minipage}{0.33\textwidth}
        \includegraphics[width=\textwidth]{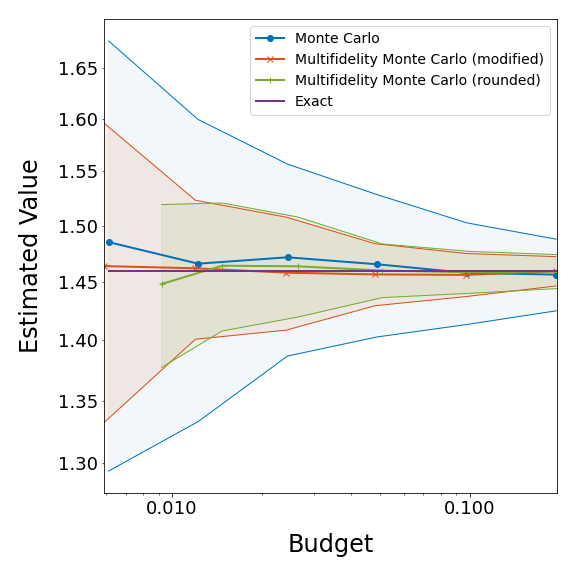}
    \end{minipage}%
    \begin{minipage}{0.33\textwidth}
        \includegraphics[width=\textwidth]{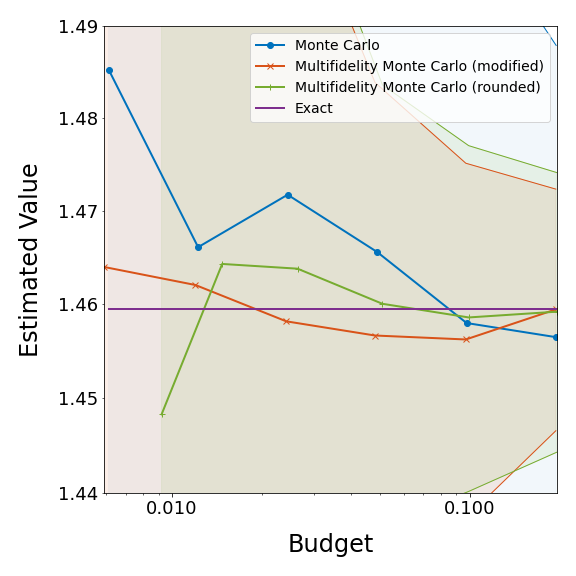}
    \end{minipage}%
    \begin{minipage}{0.33\textwidth}
        \centering
        \includegraphics[width=\textwidth]{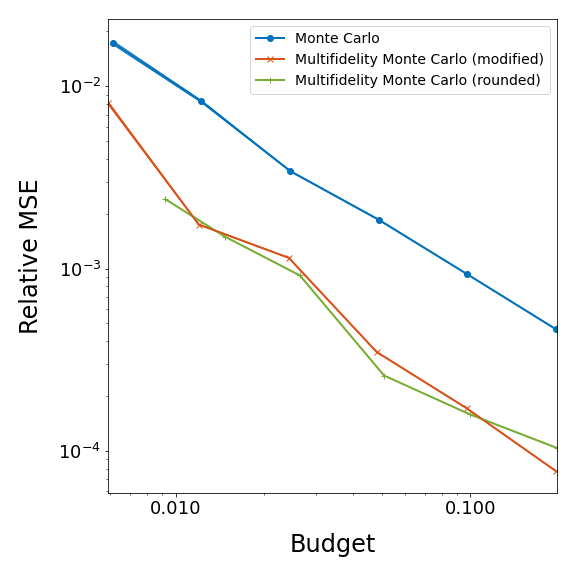}
    \end{minipage}
    \caption{Pictures corresponding to the inviscid Burgers experiment.  Plotted (left, middle) are the mean estimates and (right) MSEs over $N=100$ independent runs of each method, along with shaded regions representing the standard deviations (left, middle) and variances (right, almost imperceptible) in these quantities over the $N$ runs.}
    \label{fig:burgersROM}
\end{figure}

\section{Conclusion}
A method for the multifidelity Monte Carlo estimation of statistical quantities has been presented which is applicable to computational budgets of any size and reduces variance when compared to simple Monte Carlo estimation.  To accomplish this, existing MFMC technology from \cite{peherstorfer2016} has been adapted and modified, leading to the MFMC estimation Algorithm~\ref{alg:mod} which is fast to compute and simple to implement.  It has been shown through Theorem~\ref{thm:mod} and Corollary~\ref{cor:optim} that the proposed algorithm solves the sampling problem \eqref{eq:MFMCmin} at least as optimally as its namesake, and numerical experiments have been conducted which validate this fact. Forthcoming work in  \cite{gruber2022climate} will investigate applications of the present MFMC method to complex systems governed by partial differential equations, particularly in the context of climate modeling.  

\section*{Acknowledgements}
This work is partially supported by U.S. Department of Energy under grants DE-SC0020270, 
DE-SC0020418, and DE-SC0021077.

\bibliographystyle{ieeetr}
\bibliography{biblio.bib}

\end{document}